\documentclass[12pt]{article}
\usepackage[pdftex]{graphicx}
\usepackage{graphicx}
\usepackage[utf8]{inputenc}
\usepackage{amssymb,amsmath,amsthm}
\usepackage{anysize}
\usepackage{float}
\usepackage{multirow}
\usepackage[usenames]{color}
\usepackage{tikz}
\usepackage[colorlinks=true, citecolor=blue, linkcolor=blue, urlcolor=blue]{hyperref}
\papersize{27.9cm}{21.5cm} 
\marginsize{2.3cm}{2.3cm}{1.5cm}{2cm}
\theoremstyle{definition}
\newtheorem{teo}{Theorem}    [section]
\newtheorem{lem}{Lemma}      [section]
\newtheorem{eje}{Example}    [section]
\newtheorem{defi}{Definition}[section]
\newtheorem{pro}{Proposition}[section]

\newtheorem{obs}{Remark}     [section]
\newtheorem{coro}{Corollary} [section]
\newtheorem{cons}{Construction
} [section]
\usepackage{array}
\newcolumntype{?}{!{\vrule width 1,5pt}}
\newcommand{\G}{\mathcal{S}}
\newcommand{\Z}{{\ensuremath{\mathbb{Z}}}}

\numberwithin{equation}{section}

\newcommand{\Field}{\mathbb{F}}
\newcommand{\B}{\mathcal{B}}
\newcommand{\R}{\mathcal{R}}
\newcommand{\I}{\mathcal{I}}
\newcommand{\C}{\mathcal{C}}
\newcommand{\A}{\mathcal{A}}

	\title{\textbf{Sidon sets and $C_4$-saturated graphs}}

\author{David F. Daza \footnote{Partially supported by COLCIENCIAS No. 110371250560 and VRI - Unicauca No. 4400.}  \hspace{0.01cm} \footnote{Email: \url{daviddaza@unicauca.edu.co}}  \hspace{0.4cm}  Carlos A. Trujillo \footnotemark[1] \hspace{0.01cm} \footnote{Email: \url{trujillo@unicauca.edu.co}}\\ \vspace{-0.2cm}
 \small \textit{Departamento de Matemáticas} \\ \vspace{-0.2cm}
  \small \textit{Universidad del Cauca} \\ 
	 \small  \textit{Popayán, Colombia}\\ \and
 Fernando A. Benavides \footnote{Partially supported by VIPRI - UdeNar No. 1475.} \hspace{0.01cm} \footnote{Email: \url{fandresbenavides@udenar.edu.co}}\\ \vspace{-0.2cm}
\small \textit{Departamento de Matemáticas y Estadística}\\ \vspace{-0.2cm}
\small \textit{Universidad de Nariño}\\ \vspace{-0.2cm}
\small \textit{San Juan de Pasto, Colombia}}
	
	\date{}
	
\begin{document}

\maketitle

\begin{abstract}
\noindent
The problem of determining the Turán number of $C_4$ is a well studied problem that dates back to a paper of Erd{\H o}s from 1938. It is known that Sidon sets can be used to construct $C_4$-free graphs. If $\A$ is a Sidon set in the abelian group $X$, the sum graph $G_{X, \A}$ with vertex set $X$ and edges set $E=\{\{x, y\}:x\neq y, x+y\in \A\}$ is $C_4$-free. Using the sum graph of a Sidon set of type Singer we verify a conjecture of  Erd{\H o}s and Simonovits concerning the number of copies of $C_4$ in a graph with $ex(q^2+q+1, C_4)+1$ edges. Further, we give a sufficient condition for the sum graph of a Sidon set to be $C_4$-saturated and describe new $C_4$-saturated graphs.

\vspace{0.2cm} 

\end{abstract}

\section{Introduction}

In 1938, Erd\H{o}s asked how many edges a graph with $n$ vertices without $C_4$ may have. This number is denoted
by $ex(n, C_4)$ and is called the Turán number for $C_4$. In general, for a graph $H$ the Turán number $ex(n, H)$ is
the maximmun number of edges that a graph on $n$ vertices can have without containing a copy of $H$. A graph
on $n$ vertices with $ex(n, H)$ edges and does not contain $H$ as subgraph is called extremal. The problem of
determining the number $ex(n, H)$ belongs to the field of extremal graph theory. It is known that
$ex(n, C_4)\leq \frac{1}{2} n^{3/2}+ o(n^{3/2})$, in fact for some values of $n$ the Turán number for $C_4$ has been determined. For example, using computer searches $ex(n, C_4)$ was calculated for $n\leq 21$ by Clapham, Flockhart and Sheehan
in \cite{Clapham} and it was extended for $n\leq 31$ by Rowlinson and Yuansheng in \cite{Rowlinson}. Füredi in \cite{Furedi} proved that $ex(q^2+q+1, C_4) \geq \frac{1}{2}q (q+1)^2$ when $q$ is a power of 2. On the other hand, Brown in \cite{Brown} and Erd\H{o}s, Rényi and Sós in \cite{Erdos2} proved independently that $ex(q^2 + q + 1, C_4)\geq \frac{1}{2} q(q + 1)^2$ where $q$ is a prime power. They used the
Erd\H{o}s-Rényi graphs $ER_q$ which are derived from orthogonal polarity graphs of the projective plane $PG(2, q)$.
Later, Füredi in \cite{Furedi2} proved that $ex(q^2+q+1, C_4)\leq \frac{1}
{2} q(q+1)^2$ where $q > 13$ is a prime power and the equality
is satisfied when the graph is an orthogonal polarity graph of a projective plane of order $q$. This shows the
exact value for $n = q^2 + q + 1$ with $q > 13$ a prime power. In \cite{Abreu}, Abreu, Balbuena and Labbate, by deleting carefully chosen vertices from the Erd\H{o}s-Rényi graph, exhibited lower bounds to $ex(n, C_4)$ for certain values of $n$. For instance, they proved that

\begin{center}
$ex(q^2-q-2, C_4)\geq\left\{ \begin{array}{lcc}
             (\frac{1}{2}q-1)(q^2-1) &   \text{if} \hspace{0.2cm} q \hspace{0.2cm} \text{is odd}\\
             \\ \frac{1}{2}q^3-q^2 & \text{if} \hspace{0.2cm} q \hspace{0.2cm} \text{is even}
				\end{array}
   \right.$
\end{center}

 \noindent
where $q$ is any prime power. Another result was established by Firke, Kosek, Nash and Williford in \cite{Firke}, they proved that $ex(q^2 +q, C_4)\leq \frac{1}{2} q(q+1)^2- q$ for even $q$. Recently, Tait and Timmons in \cite{Tait} improved the above lower bound given by Abreu, Balbuena and Labbate. They constructed a $C_4$-free graph $G_{q, \theta}$ with $q^2-q-2$ vertices and at least $\frac{1}{2} q^3-q^2-O(q^{3/4})$ edges, by removing a particular subgraph from $G_{q, \theta}$. In this case, $G_{q, \theta}$ is the sum graph associated to a Sidon set of type Bose-Chowla. In the field of extremal graph theory there are still many problems unsolved, for instance in \cite{Erdos3}, Erd\H{o}s and Simonovits conjectured that if $G$ is any graph on $n$ vertices with $ex(n, C_4)$ + 1 edges, then $G$ must contain at least $n^{1/2}+o(n^{1/2})$ copies of $C_4$.

\vspace{0.2cm}

\begin{teo}\label{pr}
Let $(X, +)$ be a finite abelian group and $\A$ be a Sidon set in $X$ with zero-deficiency and $|\A|=\sqrt{|X|}-\delta$. If $H$ is a graph obtained by adding an edge to the sum graph $G_{X, \A}$ then $H$ contains at least $\sqrt{|X|}+o(\sqrt{|X|})$ copies of $C_4$. 
\end{teo}

\textbf{Results.} Note that if the number of edges of the sum graph of Theorem \ref{pr} is $ex(|X|, C_4)$ then the conjecture of Erd\H{o}s and Simonovits is verified. The unique known Sidon sets with zero deficiency are the Sidon sets of type Singer, these sets verify the hypothesis of Theorem \ref{pr} and we prove that the number of edges of the sum graph of a Sidon set of type Singer is $ex(q^2+q+1, C_4)$ whenever that $q>13$ is a prime power. Thus for these graphs we verify the conjecture of Erd\H{o}s and Simonovits. On the other hand, a graph $G$ is called $H$-saturated if $G$ does not contain a subgraph isomorphic to $H$ but the addition of an edge joining any pair of nonadjacent vertices of $G$ completes a copy of $H$. We use Sidon sets (Singer,
Ruzsa, cartesian products) to construct $C_4$-saturated graphs and for each sum graph we determine the number of copies of $C_4$ obtained
by adding an edge to each graph. Finally, we find relations between the maximality of a Sidon set and its corresponding graph associated.

\vspace{0.5cm}

\textbf{Organization.} This paper is divided in three sections, in Section 2 we include important aspects of Sidon sets defined on additive abelian groups, for example we describe important constructions of this kind of sets (Bose-Chowla, Singer, Ruzsa, etc), which we shall use to construct $C_4$-saturated graphs. In Section 3, we prove important relations between the maximality of Sidon
sets and $C_4$-saturated graphs. In Subsection 3.1, we prove that the sum graph associated to each
Sidon set is $C_4$-saturated and as a corollary of this results we prove that this kind of graphs are $C_4$-saturated if
and only if the corresponding Sidon set is maximal. Finally, in this subsection we prove Theorem \ref{pr} and with this theorem and the Sidon sets of type Singer we give a family of graphs that verify the conjecture of Erd\H{o}s and Simonovits.

\section{Sidon sets}

Let $X$ be an (additive) abelian group, a non empty subset $\A \subset X$ is a \textbf{Sidon set} in $X$ if 
\begin{equation}\label{s1}
a+b=c+d  \hspace{0.3cm} \text{implies that} \hspace{0.3cm} \{a, b\}=\{c, d\},
\end{equation}

\noindent
 for all $a, b, c , d \in \A$. We say that a Sidon set $\A$ of an abelian group $X$ is \textbf{maximal} if it is not properly contained in another Sidon set in $X$.  Sidon sets were considered in integers by Simon Sidon in \cite{Sidon1}. According to \cite{Erdos1}, Simon Sidon introduced Sidon sets to Erd\H{o}s in 1932 or 1933 and he was interested in how large a Sidon set $\A$ can be if $\A\subset \{1, 2, \ldots, n\} = [1, n]$, in other words Sidon wanted to determine the function
\begin{center}
$F_2(n)=\max\{|\A|: \A \hspace{0.1cm} \text{is a Sidon set  and} \hspace{0.1cm} \A\subset [1, n]\}$.
\end{center}

\noindent
It is known that $F_2(n)\sim \sqrt{n}$. In order to obtain lower bounds of $F_2(n)$ it is necesary to construct Sidon sets,
similarly with the number $ex(n, C_4)$. For instance in the literature we can find many constructions of Sidon
sets as Bose-Chowla in \cite{Bose}, Singer in \cite{Singer}, Ruzsa in \cite{Ruzsa} or constructions on cartesian products $\Field_p\times \Field_p, \Field_p^+\times \Field_p^*$ and $\Field_p^* \times \Field_p^*$ in \cite{Caicedo}. Today Sidon sets have been applied to many areas as communications, fault-tolerant
distributed computing, coding theory, graph theory, see \cite{Blum, Golomb, Kovacevi, Klonowska}.\\

As the equation \ref{s1} implies that $a-d =c-b$, the Sidon sets are defined as such sets with the property that all non-zero differences of elements of that set are different. By counting the number of differences $a-b$, we can see that if $\A$ is a Sidon set in $X$, then $|A|<\sqrt{|X|}+1/2$. The most interesting Sidon sets are those which have large cardinality, that is, $|\A|=\sqrt{|X|}-\delta$ where $\delta$ is a small number.\\

If $\A$ is a finite Sidon set in an abelian group $(X, +)$, then the \textbf{difference set} is
defined as usual; i.e.,

\vspace{-0.4cm}

\begin{center}
$\A-\A:=\{a-b: a, b\in \A \}$.
\end{center}
 
\noindent
Note that if $|\A|=k$ then $|\A\ominus \A|= 2 \genfrac(){0pt}{0}{k}{2}$ where $\A\ominus \A=(\A-\A)$\textbackslash$\{0\}$. \\

Moreover, if $X$ is finite we define the deficiency of the set $\A$ denoted by $d(\A)$ as the \mbox{cardinal} of the set $X$\textbackslash$\A-\A$, this is, $d(\A):=|X|-|\A-\A|$. Therefore if $d(\A)=0$ then $\A-\A=X$.\\

The following constructions of Sidon sets and their respective lemmas can be found in \cite{Caicedo, Trujillo}, these lemmas will be fundamental in the proof of Theorem \ref{teo217}. We present the proofs of the lemmas for completeness. We know that the Sidon sets of the Constructions 2.1 to 2.5 have maximal cardinality in their ambient group $X$ and that the cardinal of a Sidon set in each of these constructions (including Construction 2.6) is $\sqrt{|X|}-\delta$ for some $\delta\leq 1$.

\begin{cons}\textbf{(Bose-Chowla)} Let $q$ be a prime power, $h\geq 2$ be a integer, $(\Field_q, +, \cdot)$ be the finite field with $q$ elements and $\theta$ a primitive element of $\Field_{q^h}$. The set, 
\begin{center}
$\B:=\log_{\theta}(\theta+\Field_q):=\{\log_{\theta}(\theta+a):a\in \Field_q\}$,
\end{center}
 is a Sidon set in $(\mathbb{Z}_{q^h-1}, +)$, with $q$ elements.  
\end{cons}

\begin{lem}\label{b1}
If $\B$ is a Sidon set of type Bose-Chowla in $(\mathbb{Z}_{q^2-1}, +)$ then
\begin{center}
$\B\ominus \B=\Z_{q^2-1}$\textbackslash$M_{q+1}$, where $M_{q+1}=\{x \in \Z_{q^2-1}:x\equiv 0 \bmod (q+1)\}$.
\end{center}
\end{lem}

\begin{proof}
Since that $\B$ is a Sidon set and $|\B|=q$ then $|\B\ominus \B|=2 \genfrac(){0pt}{0}{q}{2}=q(q-1)=q^2-q$. Morevover, as $(\B\ominus \B)\cap M_{q+1}=\emptyset$ then $|\Z_{q^2-1}|-|M_{q+1}|=q^2-1-(q-1)=q^2-q=|\B\ominus \B|$.
\end{proof}

\begin{cons}\textbf{(Singer)} \label{S}
Let $\B$ be a Sidon set of type Bose-Chowla in $\mathbb{Z}_{q^3-1}$ and \mbox{$\G:=\B\bmod (q^2+q+1)$}. Then $\G_0=\G\cup\{0\}$ is a Sidon set in $\mathbb{Z}_{q^2+q+1}$ with $q+1$ elements.
\end{cons} 

\begin{obs}
In the above theorem $\G$ is a Sidon set in $\mathbb{Z}_{q^2+q+1}$ with $q$ elements. Therefore \\ $|\G\ominus\G|= 2 \genfrac(){0pt}{0}{q}{2}$. 
\end{obs}

\begin{lem}\label{Singer1}
$\G_{0}\ominus \G_{0}=\mathbb{Z}_{q^2+q+1}$\textbackslash\{0\}.
\end{lem}

\begin{proof}
Since $|\G\ominus\G|=q^2-q, |\G|=|-\G|=q$, the sets $\G-\G, \G, -\G$ are pairwise disjoint and $\G_0\ominus \G_0=(\G-\G)\cup \G\cup(-\G)$ then,
\begin{align*}
|\G_0\ominus\G_0|&=|(\G\ominus\G)\cup \G\cup(-\G)|,\\
         &=|(\G\ominus\G)|+|\G|+|(-\G)|,\\
				 &= (q^2-q)+q+q,\\
				 &= q^2+q,\\
				 &= |\mathbb{Z}_{q^2+q+1} \text{\textbackslash}\{0\}|,
\end{align*}
therefore, $\G_0\ominus\G_0=\mathbb{Z}_{q^2+q+1}$\textbackslash\{0\}. 
\end{proof}

\begin{cons}\textbf{(Ruzsa)} \label{R}
If $\theta$ is a primitive element of the finite field $\mathbb{Z}_{p}$ then, 
\begin{center}
$\mathcal{R}:=\{x\equiv ip-\theta^i(p-1)(\bmod p^2-p):1\leq i\leq p-1\}$,
\end{center}

\noindent
is a Sidon set in $\mathbb{Z}_{p^2-p}$ with $p-1$ elements. 
\end{cons}

\begin{lem} \label{Ru}
$\R\ominus \R=\mathbb{Z}_{p^2-p}$\textbackslash{($M_p\cup M_{p-1})$}, where $M_i=\{x\in \mathbb{Z}_{p^2-p}: x\equiv 0(\bmod \, i) \}$.
\end{lem}

\begin{proof}
Note that $|M_p|=p-1$, $|M_{p-1}|=p$ and $M_p\cap M_{p-1}=\{p(p-1)\}$ because $M_p$ and $M_{p-1}$ are subsets of $\mathbb{Z}_{p^2-p}$. The above implies that,  
\begin{equation*}
|M_p\cup M_{p-1}|=|M_p|+|M_{p-1}|-|M_p\cap M_{p-1}|=2(p-1).
\end{equation*}
On the other hand, since $(\R\ominus \R)\cap M_i=\emptyset$ then   
  \begin{equation*}
(\R\ominus \R)\cap(M_p\cup M_{p-1})=\emptyset.
\end{equation*}
Now as $|\R\ominus \R|=2 \genfrac(){0pt}{0}{p-1}{2}=p^2-3p+2=p^2-p-2(p-1)$ and $\R\ominus \R\subseteq \mathbb{Z}_{p^2-p}$ then\\ 

\noindent 
$\R\ominus \R=\mathbb{Z}_{p^2-p}$\textbackslash{($M_p\cup M_{p-1})$}.

\end{proof}

Let $(\Field_{p}, +)$ be the finite field with $p$ elements and $(\Field_{p}^*, \cdot)$ its multiplicative group.  


\begin{cons}\textbf{(Cartesian Product 1)}\label{C}
If $p$ is an odd prime, then \mbox{$\C:=\{(a, a^2):a\in \Field_p\}$} is a Sidon set in $(\Field_p, +)\times (\Field_p, +)$ with $p$ elements.  
\end{cons}

\begin{lem}\label{Cuadratica}
$\C\ominus\C=\Field_p\times \Field_p$\textbackslash{$\{(0, z): z\in \Field_p\}$}.																								
\end{lem}

\begin{proof}
Indeed, 
\begin{align*}
\C\ominus \C &= \{(x, x^2)-(y, y^2): x, y\in \Field_p, (x, x^2)\neq(y, y^2)\},\\  &= \{(x-y, x^2-y^2): x, y \in \Field_p, (x, x^2)\neq(y, y^2)\}. 
\end{align*}

\noindent
If $(0, z) \in \C\ominus \C$ then $(0, z)=(x-y, x^2-y^2)$, for some $x, y$ and $z$ in $\Field_p$, so $(x, x^2)=(y, y^2)$. Therefore the elements $(0, z)$, with  $z\in \Field_p$ do not belong to $\C\ominus \C$. \\
 
To see that the set $\{(0, z): z\in \Field_p\}$ contains all the elements that do not belong to $\C\ominus\C$, we do a count.\\

Since $|\Field_p\times \Field_p|=p^2, |\C|=p$, $|\{(0, z): z\in \Field_p\}|=p$ and $\C$ is a Sidon set, 
\begin{align*}
|\C\ominus\C|=2\genfrac(){0pt}{0}{p}{2}=p^2-p.
\end{align*}
This completes the proof.
\end{proof}

\begin{cons}\textbf{(Cartesian Product 2)}\label{I}
If $p$ is an odd prime, then $\I:=\{(a, a):a\in \Field_p^*\}$ is a Sidon set in $(\Field_p, +)\times (\Field_p^*, \cdot)$ with $p-1$ elements.  
\end{cons}

We denote the operation of the group $(\Field_p, +)\times (\Field_p^*, \cdot)$ by $\star$.

\begin{lem}\label{Identidad}
If $A_1=\{(0, z):z\in \Field_p^*\}$ and $A_2=\{(z, 1):z\in \Field_p^*\}$ then 
\begin{center}
$\I\ominus\I=\Field_p\times \Field_p^*$\textbackslash{$(A_1\cup A_2)$}.
\end{center}														
\end{lem}

\begin{proof}
Indeed, 
\begin{align*}
\I\ominus \I &= \{(a, a) \star(b, b)^{-1}: a, b\in \Field_p^*, (a, a)\neq(b, b)\},\\  &= \{(a-b, ab^{-1}): a, b \in \Field_p^*, (a, a)\neq(b, b)\}. 
\end{align*}
Since $(a, a)\neq(b, b)$, the elements of the difference set never have 0 in the first component, so $(a-b, ab^{-1})\neq(0, z)$, for all $z\in\Field_ p^*$. On the other hand, since $a\neq b$ and $b\in\Field_p^*$ then $ab^{-1}\neq1$ and hence the elements of the form $(z, 1)$, $z\in\Field_ p^*$ do not belong to the  difference set.

\vspace{0.2cm}

To conclude the proof, note that $|\Field_p\times\Field_p^*|=p(p-1)$, $|\I|=p-1$, $|A_1|=|A_2|=p-1$ and $|A_1\cup A_2|=2(p-1)$ since $A_1\cap A_2=\emptyset$. Therefore,
\begin{align*}
|\I\ominus\I|=2\genfrac(){0pt}{0}{p-1}{2} &= p^2-3p+2,\\
                                          &= p(p-1)-2(p-1).
\end{align*}
\end{proof}

\begin{cons}\textbf{(Cartesian Product 3)}\label{CO}
If $p$ is an odd prime and $\alpha$ is an element in $\Field_p^*$ then $\I_{\alpha}:=\{(a-\alpha, a): a \in \Field_p^*, a\neq\alpha \}$ is a Sidon set in $(\Field_p^*, \cdot)\times (\Field_p^*, \cdot)$ with $p-2$ elements.  
\end{cons}

	\begin{lem}\label{Corrimiento}
 If  $A_1=\{(1, z): z\in \Field_p^*\}$, $A_2=\{(z, 1): z\in \Field_p^*\}$ and $A_3=\{(z, z): z\in \Field_p^*\}$ then
\begin{center}
$\I_{\alpha}\ominus\I_{\alpha}=\Field_p^*\times \Field_p^*$\textbackslash{($A_1\cup A_2\cup A_3$)}.
\end{center}
 																	
\end{lem}	

\begin{proof}

Indeed, 
\begin{align*}
\I_{\alpha}\ominus\I_{\alpha}=\{((a-\alpha)(b-\alpha)^{-1}, ab^{-1}):a, b, \alpha\in \Field_p^*, a\neq\alpha\neq b, (a-\alpha, a)\neq(b-\alpha, b)\}. 
\end{align*}

Note that $(a-\alpha, a)\neq(b-\alpha, b)$ implies that $a\neq b$. Therefore $ab^{-1}\neq1$ because $b\neq0$ and $(a-\alpha)(b-\alpha)^{-1}\neq1$ since $b-\alpha\neq0$. Hence, the elements of the form $(1, z)$ with $z\in \Field_p^*$ and the elements $(z, 1)$, $z\in \Field_p^*$ do not belong to $\I_{\alpha}\ominus\I_{\alpha}$.\\

Other elements that are not in $\I_{\alpha}\ominus\I_{\alpha}$ are of the form $(z, z)$ with $z\in \Field_p^*$. To see this, suppose that,
\begin{align*}
((a-\alpha)(b-\alpha)^{-1}, ab^{-1})=(z, z),
\end{align*}
for some $z\in \Field_p^*$. Now,
\begin{align*}
a-\alpha &= ab^{-1}(b-\alpha),\\
         &= a-ab^{-1}\alpha.
\end{align*}
Therefore, $(ab^{-1}-1)\alpha=0$, which is not possible. 

\vspace{0.2cm}

As $|\Field_p^*\times\Field_p^*|=(p-1)^2$, $|\I_{\alpha}|=p-2$ and $|A_1\cup A_2\cup A_3|=3(p-1)-2$ since $|A_1|=|A_2|=|A_3|=p-1$ and $A_1\cap A_2=A_1\cap A_3=A_2\cap A_3=A_1\cap A_2\cap A_3=\{(1,1)\}$ then,
\begin{align*}
|\I_{\alpha}\ominus\I_{\alpha}|&=2\genfrac(){0pt}{0}{p-2}{2},\\
&=p^2-5p+6,\\
&= p^2-2p+1-3p+3+2,\\
&= (p-1)^2-(3(p-1)-2).
\end{align*}
This completes the proof. 
\end{proof}	

In the following table we present the deficiency of the previous Sidon sets. Note that the unique Sidon sets with zero deficiency are the Sidon sets of type Singer. Remember that $q$ is a prime power and $p$ is an odd prime.

\begin{table}[ht]
\centering
\begin{tabular}{|c|c|c|c|c|c|}
\hline
  $X$          &  $|X|$                     & $\A$    & $|\A|$ & $|\A-\A|$ & $d(\A)=|X|-|\A-\A|$ \\\hline \noalign{\hrule height 1pt}
$\Z_{q^2+q+1}$ & $q^2+q+1$                  & $\G_0$  & $q+1$ &    $q^2+q+1$ &        0                         \\[0.2cm] \hline
$\Z_{q^2-1}$   & $q^2-1$                    & $\B$    & $q$   &    $q^2-q+1$ &        $q-2$                        \\[0.2cm] \hline
$\Z_{p^2-p}$   & $p^2-p$                    & $\R$    & $p-1$ &    $p^2-3p+3$ &      $2p-3$                        \\[0.2cm] \hline
$(\Field_p, +)\times (\Field_p, +)$ &$p^2$  &$\C$     & $p$   &     $p^2-p+1$    &       $p-1$                        \\[0.2cm] \hline
$(\Field_p, +)\times (\Field_p^*, \cdot)$ &$p^2-p$ &$\I$  & $p-1$  & $p^2-3p+3$   &     $2p-3$                       \\[0.2cm] \hline
$(\Field_p^*, \cdot)\times (\Field_p^*, \cdot)$&$(p-2)^2$ &$\I_{\alpha}$ & $p-2$   & $p^2-5p+7$ &   $3p-6$           \\ \hline
\end{tabular}
\end{table} 

\section{$C_4$-saturated Graphs on Sidon sets}


\subsection*{General Properties}

\begin{defi}
Let $X$ be an (additive) abelian group and $\mathcal{A}$ a subset of $X$, the sum graph $G_{X, \mathcal{A}}=(V, E)$ is formed by  $V=X$ and $\{x, y\}\in E$ if $x+y\in \mathcal{A}$ with $x\neq y$. A vertex $z$ is called absolute if $z+z \in \A$, let $P$ be the set of all absolute vertices.
\end{defi}

It is known that if $\A$ is a Sidon set in $X$, the sum graph $G_{X, \mathcal{A}}$  is $C_4$-free, to see this let $(x_0, x_1, x_2, x_3)$ be a $C_4$ in $G_{X, \mathcal{A}}$ (see Figure \ref{f1} ), then $x_0+x_1=a_1$, $x_1+x_2=a_2$, $x_2+x_3=a_3$ and $x_3+x_0=a_4$ where $a_1, a_2, a_3, a_4 \in \A$. Hence,

\begin{center}
$(x_0+x_1)+(x_2+x_3)=a_1+a_3=a_2+a_4=(x_1+x_2)+(x_3+x_0)$
\end{center}
Thus, $\{a_1, a_3\}=\{a_2, a_4\}$. If $a_1=a_2$ or $a_1=a_4$ then $x_0=x_2$ or $x_1=x_3$.       

\begin{figure}[H]
\centering
\begin{tikzpicture}
\draw (-2.,0.)-- (0.,0.);
\draw (0.,0.)-- (0.,2.);
\draw (0.,2.)-- (-2.,2.);
\draw (-2.,2.)-- (-2.,0.);

\draw (-1.3,2.6) node[anchor=north west] {$a_4$};
\draw (0.1,2.3) node[anchor=north west] {$x_3$};
\draw (-2.8,2.3) node[anchor=north west] {$x_0$};
\draw (-2.8,1.3) node[anchor=north west] {$a_1$};
\draw (0.1,0.3) node[anchor=north west] {$x_2$};
\draw (-2.8,0.3) node[anchor=north west] {$x_1$};
\draw (-1.3,-0.1) node[anchor=north west] {$a_2$};
\draw (0.1,1.3) node[anchor=north west] {$a_3$};
\begin{scriptsize}
\draw [fill=cyan] (-2.,0.) circle (3.0pt);
\draw [fill=cyan] (0.,0.) circle (3.0pt);
\draw [fill=cyan] (-2.,2.) circle (3.0pt);
\draw [fill=cyan] (0.,2.) circle (3.0pt);
\end{scriptsize}
\end{tikzpicture}

\vspace{-0.3cm}

\caption{Cycle $C_4$ in $G_{X, A}$} \label{f1}

\end{figure}
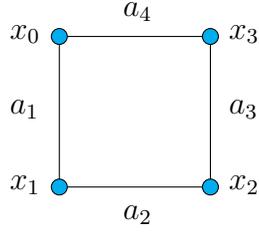

\begin{pro} \label{pro218} Let $X$ be an (additive) finite abelian group and $\A$ be a subset of $X$. \\
 If $G_{X, \A}=(V, E)$ is the sum graph of $\A$ then

\begin{center}
$2|E|=|X||\A|-|P|$
\end{center}
   \end{pro}
	
	\begin{proof}
	Let $x$ be a vertex, then deg($x$)=$|\A|$-1 if $x\in P$ or deg($x$)=$|\A|$ in other case. Therefore,
		
		\begin{center}
		$2|E|=\displaystyle\sum_{x\in P}{\text{deg}(x)}+\displaystyle\sum_{x\notin P}{\text{deg}(x)}=(|\A|-1)|P|+(|X|-|P|)|\A|=|X||\A|-|P|$
		\end{center}
	\end{proof}

	\begin{defi} Let $X$ be an (additive) abelian group and $\A$ be a Sidon set in $X$. For each $z\in X$\textbackslash$\A$ let
		\begin{center}
	$T(z):=\{(a_1, a_2, a_3)\in \A^3 : z=a_1-a_2+a_3\}$.
	\end{center}
		\end{defi}
		
		Note that if $(a_1, a_2, a_3)\in T(z)$ then $a_1\neq a_2$ and $a_2\neq a_3$. Moreover, two  distinct tuples in $T(z)$ have their first or third coordinate different. Indeed, if $(a_1, a_2, a_3)$, $(a_1, a_4, a_5) \in T(z)$ then $a_1-a_2+a_3=a_1-a_4+a_5$, so $a_3+a_4=a_5+a_2$, thus $a_2=a_4$ and $a_3=a_5$ since $\A$ is a Sidon set in $X$ and $a_2\neq a_3$. A similar argument applies to the third coordinate of each two such tuples.\\
		
	 On the other hand, if $x, y \in X$ are distinct elements such that $x+y\notin \A$ then every tuple $(a_1, a_2, a_3)\in T(x+y)$ generate a path of length at most 3 where the vertices are given by $x_{i-1}+x_i=a_i$ with $x_0=x$ and $x_3=y$. In the Figure \ref{f2} is shown the different cases.

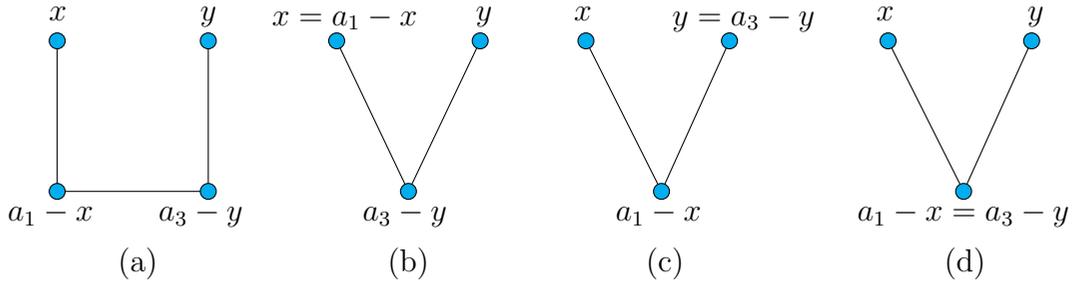
\begin{figure}[H]
\centering
\begin{tikzpicture}
\draw (-2.,2.)-- (0.,2.);
\draw (-2.,4.)-- (-2.,2.);
\draw (0.,2.)-- (0.,4.);

\draw (-0.25,4.6) node[anchor=north west] {$y$};
\draw (-2.25,4.6) node[anchor=north west] {$x$};
\draw (-0.8,2.0) node[anchor=north west] {$a_3-y$};
\draw (-2.8,2.0) node[anchor=north west] {$a_1-x$};
\draw (-1.34,1.4) node[anchor=north west] {(a)};

\begin{scriptsize}
\draw [fill=cyan] (-2.,2.) circle (3.0pt);
\draw [fill=cyan] (0.,2.) circle (3.0pt);
\draw [fill=cyan] (0.,4.) circle (3.0pt);
\draw [fill=cyan] (-2.,4.) circle (3.0pt);
\end{scriptsize}

----------------------------------------------

\draw (1.7,4.)-- (2.65,2.);
\draw (2.65,2.)-- (3.6,4.);

\draw (3.4,4.6) node[anchor=north west] {$y$};
\draw (0.7,4.6) node[anchor=north west] {$x=a_1-x$};
\draw (1.9,2.0) node[anchor=north west] {$a_3-y$};
\draw (2.23,1.4) node[anchor=north west] {(b)};

\begin{scriptsize}
\draw [fill=cyan] (1.7,4.) circle (3.0pt);
\draw [fill=cyan] (2.65,2.) circle (3.0pt);
\draw [fill=cyan] (3.6,4.) circle (3.0pt);
\end{scriptsize}

-------------------------------------------

\draw (5,4.)-- (6,2.);
\draw (6,2.)-- (6.9,4.);

\draw (6.0,4.6) node[anchor=north west] {$y=a_3-y$};
\draw (4.7,4.6) node[anchor=north west] {$x$};
\draw (5.25,2.0) node[anchor=north west] {$a_1-x$};
\draw (5.65,1.4) node[anchor=north west] {(c)};

\begin{scriptsize}
\draw [fill=cyan] (5,4.) circle (3.0pt);
\draw [fill=cyan] (6,2.) circle (3.0pt);
\draw [fill=cyan] (6.9,4.) circle (3.0pt);
\end{scriptsize}
------------------------------------

\draw (9,4.)-- (10,2.);
\draw (10,2.)-- (10.9,4.);

\draw (10.7,4.6) node[anchor=north west] {$y$};
\draw (8.7,4.6) node[anchor=north west] {$x$};
\draw (8.45,2.0) node[anchor=north west] {$a_1-x=a_3-y$};
\draw (9.6,1.4) node[anchor=north west] {(d)};

\begin{scriptsize}
\draw [fill=cyan] (9,4.) circle (3.0pt);
\draw [fill=cyan] (10,2.) circle (3.0pt);
\draw [fill=cyan] (10.9,4.) circle (3.0pt);
\end{scriptsize}
\end{tikzpicture}

\vspace{-0.3cm}

\caption{Paths for a tuple in $T(x+y)$} \label{f2}
\end{figure}

\begin{lem}\label{T1}
Let $X$ be an (additive) abelian group, $\A$ be a Sidon set in $X$ and $x, y$ be two different elements in $X$. If $x+y\notin A$ and $T(x+y)$ is non empty then there exist at most one tuple $(a_1, a_2, a_3)\in T(x+y)$ which satisfies $x=a_1-x$ or $y=a_3-y$ or $a_1-x=a_3-y$. 
					 \end{lem}
					
\begin{proof}
Let $(a_1, a_2, a_3), (a_4, a_5, a_6)$ be two distinct tuples in $T(x+y)$.\\
\textbf{Case 1.} If $x=a_1-x$ and $x=a_4-x$, then $a_1=a_4$, which is not possible.\\
\textbf{Case 2.} If $y=a_3-y$ and $x=a_6-y$, then $a_3=a_6$, which is not possible.\\
\textbf{Case 3.} If $a_1-x=a_3-y$ and $a_4-x=a_6-y$ then $a_4-a_1+a_3-y=a_6-y$, so $a_4+a_3=a_6+a_1$. As $\A$ is a Sidon set and the two tuples are distinct then $a_1=a_3$ and $a_4=a_6$. This last implies that $a_1-x=a_1-y$ and therefore $x=y$, which is a contradiction.    
  					\end{proof}
						
						\begin{lem} \label{d1}
		If $A$ is a finite Sidon set in the abelian group $(X, +)$ then  $0\leq|T(z)|\leq|A|$ for all $z\in X$\textbackslash$\A$.
				\end{lem}
				\begin{proof}
Suppose that $\A=\{a_1, a_2,\cdots, a_k\}$ and $z\in X$\textbackslash$\A$. If $z-a_i \notin \A\ominus \A$ for all $i$, then $|T(z)|=0$. Now, If $z-a_i \in \A\ominus \A$ for $1<i<k$ then $0<|T(z)|<|\A|$ and if $z-a_i \in \A\ominus \A$ for all $i$ then there are two unique (uniqueness is because $\A$ is a Sidon set) $a_j, a_k \in \A$ such taht $z-a_i=a_j-a_k$, this is, $z=a_i-a_k+a_j$. The above implies that each $a_i$ produces a different tuple. As there are $|\A|$ choices for $a_i$ then $|T(z)|=|\A|$. Therefore $0\leq|T(z)|\leq|\A|$ for all $z\in X$\textbackslash$\A$. 
				\end{proof}
				
				Note the following relationship between $d(A)$ and $T(z)$ when $A$ is a Sidon set in the finite abelian group $(X, +)$.
				
				\begin{itemize}
					\item [a.)] $d(\A)=|X|$ if and only if $|T(z)|=0$,
					\item [b.)] $0<d(\A)<|X|$ if and only if $0<|T(z)|<|\A|$,
					\item [c.)] $d(\A)=0$ if and only if $|T(z)|=|\A|$.
				\end{itemize}
				
					\begin{teo}\label{tuplas}
Let $X$ be an (additive) abelian group and $\A$ be a finite Sidon set in $X$. If for all $z \in X$\textbackslash$\A$, $|T(z)|\geq 4$ then $G_{X, \A}$ is $C_4$-saturated. 
\end{teo}

\begin{proof}
Let $x, y$ be different elements in $X$ such that $x+y \notin \A$ and $m=|T(x+y)|$. By Lemma \ref{T1} there are at least $m-3$ tuples $(a_1, a_2, a_3)$ in $T(x+y)$ such that $x\neq a_1-x$, $y\neq a_3-y$ and $a_1-x\neq a_3-y$, i.e. that these tuples do not generate the paths (b), (c) and (d) of Figure \ref{f2} but if the path (a) of Figure \ref{f2}. Thus there are at least $m-3$ paths of length three between $x$ and $y$. So adding the edge $xy$ gives a 4-cycle. Therefore $G_{X, \A}$ is $C_4$-saturated. 
\end{proof}

\begin{obs} \label{cycles} Under the assumptions of Theorem \ref{tuplas}, if $H$ is a graph obtained by adding an edge to $G_{X, \A}$ then $H$ contains at least $|T(z)|-3$ copies of $C_4$.
\end{obs}

\begin{pro} Let $X$ be an (additive) abelian group and $\A$ be a Sidon set in $X$. If for all $z\in X$\textbackslash$\A$, $T(z)\neq \emptyset$ then $\A$ is maximal.
\end{pro}

\begin{proof}
By contradiction suppose $\A$ is not maximal. Then there exists $z \in X$\textbackslash$\A$ such that $\A'=\A\cup \{z\}$ is a Sidon set in $X$. If there exists $(a_1, a_2, a_3) \in \A^3$ such that $z=a_1-a_2+a_3$ then $\{z, a_2\}=\{a_1, a_3\}$ since $\A'$ is a Sidon set. This implies that $z \in \A$ which is a contradiction.     
\end{proof}

\begin{teo}\label{maximal} Let $X$ be an (additive) abelian group and $\A$ be a Sidon set in $X$. If $G_{X, \A}$ is $C_4$-saturated then $\A$ is maximal.
\end{teo}
  
	\begin{proof}
	By contradiction suppose $\A$ is not maximal. Then there exists $z\in X$\textbackslash$\A$ such that $\A'=\A\cup \{z\}$ is a Sidon set in $X$, which implies that $G_{X, \A}$ is $C_4$-free. On the other hand,

	\begin{center}
	$|\{(x, y)\in X\times X:x+y=z\}|\geq 1$
	\end{center}
	thus, $E(G_{X, \A})\subset E(G_{X, \A'})$ and so $G_{X, \A}$ is not $C_4$-saturated. 
	\end{proof}

\subsection{Sum Graph of Sidon Set}

In this section we study the sum graph associated to some constructions of Sidon sets, which are $C_4$-saturated.
First we consider a Singer Sidon set $\G_0$, which is a subset of the additive abelian group $X = \Z_{q^2+q+1}$ where $q$
is prime power with $|\G_0|=q + 1$.

\begin{lem}\label{lem213}
If $P$ is the set of absolute vertices of $G_{X, \G_0}$ then $|P|=q+1$.
\end{lem}

\begin{proof}
Since $q^2+q+1$ is odd, for each $a\in \G_0$ the equation $x+x\equiv a \bmod (q^2+q+1)$ has unique solution. Then $|P|=q+1.$
\end{proof}

\begin{teo}\label{ex}
The number of edges in $G_{X, \G_0}$ is $\frac{1}{2}q(q+1)^2$.
\end{teo}

\begin{proof}
By Proposition \ref{pro218} and Lemma \ref{lem213}

\begin{center}
$|E|=\frac{1}{2}[(q^2+q+1)(q+1)-(q+1)]=\frac{1}{2}(q+1)(q^2+q)=\frac{1}{2}q(q+1)^2$
\end{center}
\end{proof}

\begin{coro}\label{extremal}
For each prime power $q>13$, $ex(q^2+q+1, C_4)=\frac{1}{2}q(q+1)^2$.
\end{coro}

\begin{proof}
Theorem \ref{ex} implies that $ex(q^2+q+1, C_4)\geq\frac{1}{2}q(q+1)^2$. On the other hand, by Theorem 1 de \cite{Furedi2}, $ex(q^2+q+1, C_4)\leq\frac{1}{2}q(q+1)^2$.  
\end{proof}

\begin{obs}
In \cite{Timmons} it is established that $ex(q^2+q, C_4)\geq  \frac{1}{2}q(q+1)^2-q$ if $q$ is a power of two. This bound can be obtained by  removing a vertex of degree $q$ of the sum graph of a Sidon set of type Singer, if $q$ is a power of two. In \cite{Firke} it is established that when $q$ is a power of two, the free graphs of $C_4$ with $q^2+q$ vertices and $\frac{1}{2}q(q+1)^2-q$ edges are obtained by eliminating a vertex of degree $q$ from an orthogonal polarity graph.      
\end{obs}

In the following table we present similar results that we obtained for the Sidon sets $\B, \R, \I$ and $\I_{\alpha}$. In this table $n$ is the cardinal of the ambient group of each of the Sidon sets, $|P|$ is the cardinal of the set of absolute points, $|E|$ is the number of edges of the sum graph associated with the Sidon set. In the case of a Sidon set of type Bose-Chowla the number of absolute points is $q$ when $q$ is an even prime power and $q-1$ when $q$ is an odd prime power. In the other cases, $p$ is an odd prime.   

\begin{table}[ht]
\centering
\begin{tabular}{|l|c|c|c|c|}
\hline
   &  $|P|$                     & $|E|   \leq    ex(n, C_4)$ \\\hline \noalign{\hrule height 1pt}
$\B$     & $q$                      & $\frac{q^3-2q}{2}$              $\leq                ex(q^2-1, C_4)$    \\[0.2cm] \hline
$\mathcal{B}$     & $q-1$                        & $\frac{q^3-2q+1}{2}$                  $\leq               ex(q^2-1, C_4)$      \\[0.2cm] \hline
$\mathcal{R}$     & $p-1$                      & $\frac{p^3-2p^2+1}{2}$                 $\leq            ex(p^2-p, C_4)$       \\[0.2cm] \hline
$\mathcal{C}$     & $p$                        & $\frac{p^3-p}{2}$                              $\leq   ex(p^2, C_4)$           \\[0.2cm]\hline
$\mathcal{I}$     & $p-1$                      & $\frac{p^3-2p^2+1}{2}$                   $\leq          ex(p^2-p, C_4)$        \\[0.2cm] \hline
$\mathcal{I}_a$   & $p-4-(-1)^{\frac{p-1}{2}}$ & $\frac{p^3-4p^2+4p+2+(-1)^{\frac{p-1}{2}}}{2}$  $\leq  ex((p-1)^2, C_4)$        \\[0.2cm]\hline
 \end{tabular}
\end{table} 

\begin{lem}\label{triples}
Let $X$ be an (additive) abelian group and $\A$ be a Sidon set in $X$. For each $z\in X$\textbackslash$\A$
\begin{align*}
 \text{(i)} \hspace{0.1cm} &|T(z)|\geq q-1 \hspace{0.1cm} \text{if} \hspace{0.1cm} X=\Z_{q^2-1} \hspace{0.1cm} \text{and} \hspace{0.1cm} A=\B. \hspace{0.1cm}&   \text{(iv)} \hspace{0.1cm} &|T(z)|\geq p-1 \hspace{0.1cm} \text{if} \hspace{0.1cm} X=\Field_p\times \Field_p \hspace{0.1cm} \text{and} \hspace{0.1cm} A=\C. \hspace{0.1cm}&\\
\text{(ii)} \hspace{0.1cm} &|T(z)|= q+1 \hspace{0.1cm} \text{if} \hspace{0.1cm} X=\Z_{q^2+q+1} \hspace{0.1cm} \text{and} \hspace{0.1cm} A=\G_0. \hspace{0.1cm}& \text{(v)} \hspace{0.1cm} &|T(z)|\geq p-3 \hspace{0.1cm} \text{if} \hspace{0.1cm} X=\Field_p\times \Field_p^* \hspace{0.1cm} \text{and} \hspace{0.1cm} A=\I. \hspace{0.1cm}&\\
\text{(iii)} \hspace{0.1cm} &|T(z)|\geq p-3 \hspace{0.1cm} \text{if} \hspace{0.1cm} X=\Z_{p^2-p} \hspace{0.1cm} \text{and} \hspace{0.1cm} A=\R. \hspace{0.1cm}& \text{(vi)} \hspace{0.1cm} &|T(z)|\geq p-5 \hspace{0.1cm} \text{if} \hspace{0.1cm} X=\Field_p^*\times \Field_p^* \hspace{0.1cm} \text{and} \hspace{0.1cm} A=\I_{\alpha}. \hspace{0.1cm}&\\ 						
\end{align*}  
\end{lem}

\begin{proof} 

\begin{itemize}
\item [(i)] Suppose that $\B=\{a_1, a_2,\cdots, a_{q}\}$. Let $z$ be an element in $\Z_{q^2-1}$\textbackslash$\B$, then by Lemma \ref{b1} there exists at most one element $a_1 \in \B$ such that $z-a_1 \in M_{q-1}$. This implies that $z-a_i\in \B\ominus \B$ for all $2\leq i \leq q$, as $|\B|=q$ then $q-1\leq|T(z)|\leq q$. 

\item [(ii)] Let $z$ be an element in $\Z_{q^2+q+1}$\textbackslash$\G_0$, by Lemma \ref{Singer1} $z-a\in \G_0\ominus \G_0$ for all $a\in \G_0$, as $|\G_0|=q+1$ then $|T(z)|=q+1$. 
 
	\item [(iii)] 
Suppose that $\R=\{a_1, a_2,\cdots, a_{p-1}\}$. Let $z$ be an element in $\Z_{p^2-p}$\textbackslash$\R$, then by Lemma \ref{Ru} there exists at most one element $a_1\in \R$ such that $z-a_1 \in M_p$, similarly there exists at most one element $a_2 \in \R$ such that $z-a_2 \in M_{p-1}$. This implies that $z-a_i\in \R\ominus \R$, for all $3\leq i \leq p-1$. In addition, note that there exists a unique pair $b_i, c_i \in \R$ such that such that $z-a_i=c_i-b_i$. Therefore, $p-3\leq |T(z)|\leq p-1.$  
  
\item [(iv)] Suppose that $\C=\{(a_0, a_0^2), (a_1, a_1^2),\cdots, (a_{p-1}, a_{p-1}^2) \}$. If $(x, y)\in \Field_p\times\Field_p$\textbackslash${\C}$, then by Lemma \ref{Cuadratica} $(x, y)-(a_k, a_k^2) \in \C\ominus \C$, $0\leq k\leq p-1$, unless $x-a_k\equiv 0 \bmod p$, that is, if
 $x=a_k$ for some $k$.  Without loss of generality suppose that $x=a_0$. So, $(x, y)-(a_k, a_k^2) \in \C\ominus \C$, for all $1\leq k\leq p-1$. Therefore, $p-1\leq|T(z)|\leq p.$

\item [(v)] Let $(x, y)\in \Field_p\times\Field_p^*$\textbackslash${\I}$ \, and \,  $(a, a)\in \I$. Note that by Lemma \ref{Identidad} $(x, y)\star(a, a)^{-1}\in \I\ominus \I$ unless,
\begin{center}
$(x-a, ya^{-1})=(0,t), t\in \Field_p^*$ \hspace{0.2cm} or \hspace{0.2cm} $(x-a, ya^{-1})=(s, 1), s\in \Field_p^*$.
\end{center}
The first occurs only if $x=a$, so there are $p-2$ choices for $a$ because $a\in \Field_p^*$. The second occurs if $ya^{-1}=1$, that is, if
 $y=a$, hence there are $p-3$ choices for $a$ which guarantee that $(x, y)\star(a, a)^{-1}\in \I\ominus \I$, so there are unique $(b, b), (c, c)\in \I$ such that $(x, y)\star(a, a)^{-1}=(c, c)\star(b, b)^{-1}$. In this case there are at least $p-3$ tuples that satisfy the requirement. Thus, $|p-3\leq |T(z)|\leq p-1|$ 

\item [(vi)] Let $(x, y)\in X$\textbackslash${\I_{\alpha}}$ and $(a-\alpha, a)\in I_{\alpha}$. By Lemma \ref{Corrimiento} $(x, y)(a-\alpha, a)^{-1} \in \I_{\alpha}\ominus \I_{\alpha}$, unless,
\begin{align*}
(x, y)(a-\alpha, a)^{-1}  &= (1, z), z\in \Field_p^*,\\
(x, y)(a-\alpha, a)^{-1}  &= (s, 1), s\in \Field_p^*, \hspace{0.3cm} \text{or}\\
 (x, y)(a-\alpha, a)^{-1} &= (t, t), t\in \Field_p^*;     
\end{align*}
that is, if $x\equiv (a-\alpha) \bmod p$, $y\equiv a \bmod p$ or $xy^{-1}\equiv (1-a^{-1}\alpha)\bmod p$. These congruences have a unique solution. From the above and the fact that $|I_{\alpha}|=p-2$ it follows that there are at least $p-5$  elements $(a-\alpha, a)$ in $\I_\alpha$ with $(x, y)(a-\alpha, a)^{-1}\in \I_\alpha\ominus \I_\alpha$. Therefore, $p-5\leq |T(z)|\leq p-2$.
\end{itemize}  \end{proof}

According to Theorem \ref{tuplas} if we prove the existence of at least 4 tuples in $T(z)$ for any $z\in X$\textbackslash${\A}$ then the corresponding sum
graph $G_{X, \A}$ is $C_4$-saturated. Then Lemma \ref{triples} implies the following theorem.

\begin{teo}\label{teo217}
In each of the following cases the sum graph $G_{X, \A}$ is $C_4$-saturated
\begin{align*}
 \text{(i)} \hspace{0.1cm} & X=\Z_{q^2-1} \hspace{0.1cm} \text{and} \hspace{0.1cm} A=\B. \hspace{1.5cm}&   \text{(iv)} \hspace{0.1cm} & X=\Field_p\times \Field_p \hspace{0.1cm} \text{and} \hspace{0.1cm} A=\C. \hspace{0.1cm}&\\
\text{(ii)} \hspace{0.1cm} & X=\Z_{q^2+q+1} \hspace{0.1cm} \text{and} \hspace{0.1cm} A=\G_0. \hspace{1.5cm}& \text{(v)} \hspace{0.1cm} & X=\Field_p\times \Field_p^* \hspace{0.1cm} \text{and} \hspace{0.1cm} A=\I. \hspace{0.1cm}&\\
\text{(iii)} \hspace{0.1cm} & X=\Z_{p^2-p} \hspace{0.1cm} \text{and} \hspace{0.1cm} A=\R. \hspace{1.5cm}& \text{(vi)} \hspace{0.1cm} & X=\Field_p^*\times \Field_p^* \hspace{0.1cm} \text{and} \hspace{0.1cm} A=\I_{\alpha}. \hspace{0.1cm}&\\ 						
\end{align*} \end{teo}
		
\vspace{-1.0cm}
		
		\begin{coro}
		For each Sidon set $\A$ of type Bose-Chowla, Singer, Ruzsa, Cartesian Product 1,2 and 3. $G_{X, \A}$ is
$C_4$-saturated if and only if $\A$ is maximal.
		\end{coro}
		
		\begin{proof}
		Theorem \ref{maximal} proves that if $G_{X, \A}$ is $C_4$-saturated then $\A$ is maximal. In the other case, consider $z$
an element in the corresponding group such that $z \notin \A$. Then Lemma \ref{triples} implies that there exists a tuple in
$T(2z-z)$ such that generates a path of length 3 between $2z$ and $-z$ in $G_{X, \A}$, therefore $G_{X, \A}$ is $C_4$-saturated.
		\end{proof}
		
\noindent
\textbf{\textit{Proof of Theorem \ref{pr}}.} As $d(\A)=0$ then $\A-\A=X$, this implies that if $z\in X$\textbackslash${\A}$ then $z-a\in \A\ominus\A$ for all $a \in \A$. Thus $|T(z)|=|\A|=\sqrt{|X|}-\delta$ and by Remark \ref{cycles} if $H$ is a graph obtained by adding an edge to $G_{X, \A}$ then $H$ contains at least $|T(z)|-3=\sqrt{|X|}-\delta-3$ copies of $C_4$. $\hfill\square$ 
  			
\begin{eje}
Let $q$ be a prime power and $X =\Z_{q^2+q+1}$. By Lemma \ref{Singer1} $d(\G_0)=0$ and by Lemma \ref{triples} $|T(z)|=q+1$ for all $z\in X$\textbackslash${\G_0}$. Thus by Theorem \ref{pr} if $H$ is a graph obtained by adding an edge to the sum graph $G_{X, \G_0}$ then $H$ contains at least $|T(z)|-3=(q+1)-3=q-2$ copies of $C_4$, this is, $H$ contains at least $\sqrt{|X|}+o(\sqrt{|X|})$ copies of $C_4$. On the other hand, by Theorem \ref{ex} the number of edges of the graph $G_{X, \G_0}$ is $ex(q^2+q+1,C_4)$. Hence the conjecture of Erd\H{o}s and Simonovits is satisfied in $G_{X, \G_0}$.
\end{eje}

		\begin{coro}
		Let $\A$ be a Sidon set of type Bose-Chowla, Singer, Ruzsa, Cartesian Product 1,2 and 3. If $H$ is a
graph obtained by adding an edge to $G_{X, \A}$ then $H$ contains at least
\begin{align*}
 \text{(i)} \hspace{0.1cm} & q-4 \hspace{0.1cm} \text{copies of} \hspace{0.1cm} C_4 \hspace{0.1cm} \text{if} \hspace{0.1cm} A=\B. \hspace{1.5cm}&   \text{(iv)} \hspace{0.1cm} & p-6 \hspace{0.1cm} \text{copies of} \hspace{0.1cm} C_4 \hspace{0.1cm} \text{if} \hspace{0.1cm} A=\I.   \hspace{0.1cm}&\\
\text{(ii)} \hspace{0.1cm} & p-6 \hspace{0.1cm} \text{copies of} \hspace{0.1cm} C_4 \hspace{0.1cm} \text{if} \hspace{0.1cm} A=\R. \hspace{1.5cm}& \text{(v)} \hspace{0.1cm} & p-8 \hspace{0.1cm} \text{copies of} \hspace{0.1cm} C_4 \hspace{0.1cm} \text{if} \hspace{0.1cm} A=\I_{\alpha}.  \hspace{0.1cm}&\\
 \text{(iii)} \hspace{0.1cm} & p-4 \hspace{0.1cm} \text{copies of} \hspace{0.1cm} C_4 \hspace{0.1cm} \text{if} \hspace{0.1cm} A=\C.  \hspace{0.1cm}&\\ 						
\end{align*} 
		\end{coro}
		
		\vspace{-1.0cm}
			
		\section*{Acknowledgement}
The authors would like to thank Master and Ph. D programs in Mathematics of Universidad del Cauca (Colombia).
Also first two authors would like to thank COLCIENCIAS and Universidad del Cauca for supporting
the research project ``Aplicaciones a la teoría de la información y comunicación de los Conjuntos de Sidon y
sus generalizaciones'' (Códigos 110371250560, VRI - 4400). The third author would like to thank VIPRI -
Universidad de Nariño for supporting the research project ``Subálgebras de Mishchenko-Fomenko en $U(gl_{n})$ y
secuencias regulares'' (Código 1475).
		
		\addcontentsline{toc}{chapter}{Bibliografía}
\markboth{Bibliografía}{Bibliografía}
\bibliographystyle{unsrt}

  \end {document}